\documentclass[a4paper,12pt]{amsart}
\usepackage[utf8]{inputenc}
\usepackage{amsfonts,amssymb,amsthm,amsmath,color}

\theoremstyle{plain}
\newtheorem{theorem}{Theorem}[section]

\newtheorem{proposition}[theorem]{Proposition}

\newtheorem*{problemintro*}{Problem}

\theoremstyle{definition}

\newtheorem{problem}[theorem]{Problem}

\DeclareMathOperator{\lk}{lk} 

\newcommand{\poly}{\ensuremath{\operatorname{poly}}}
\newcommand{\cc}{\ensuremath{\mathbf{c}}}
\newcommand{\dd}{\ensuremath{\mathbf{d}}}
\newcommand{\aaa}{\ensuremath{\mathbf{a}}}
\newcommand{\bb}{\ensuremath{\mathbf{b}}}
\newcommand{\flag}{\ensuremath{\rm{flag}}}
\newcommand{\Cone}{\ensuremath{\operatorname{Cone}}}
\newcommand{\Conv}{\ensuremath{\operatorname{Conv}}}

\title{Complexity yardsticks for $f$-vectors of polytopes and spheres}

\author{Eran Nevo}
\thanks{}

\begin{document}
\maketitle
\begin{center}
\emph{Dedicated to the memory of Branko Gr\"{u}nbaum}
\end{center}

\begin{abstract}
We consider geometric and computational measures of complexity for sets of integer vectors, asking for a qualitative difference between $f$-vectors of simplicial and general $d$-polytopes, as well as flag $f$-vectors of $d$-polytopes and regular CW $(d-1)$-spheres,
for $d\ge 4$.
\end{abstract}

\section{Introduction}
The face numbers of simplicial $d$-polytopes are characterized by the celebrated $g$-theorem, conjectured by McMullen~\cite{McMullen:NumberFaces-71} and proved by Stanley~\cite{Stanley:NumberFacesSimplicialPolytope-80} and Billera-Lee~\cite{BilleraLee:SufficiencyMcMullensConditions-1981}. In contrast, the $f$-vector, and the finer flag $f$-vector, of general $d$-polytopes of dimension $d \geq 4$ are not well understood, despite considerable effort, see e.g. Gr\"{u}nbaum's book~\cite[Ch.10]{Grunbaum:ConvexPolytopes-03}; likewise for regular and strongly regular CW spheres. Are there ``qualitative" differences between these sets of vectors? In this note we suggest geometric measures to make this question precise. The computational complexity aspect is also considered.
For other measures of complexity in dimension $4$, like \emph{fatness}, see e.g. Ziegler's ICM paper~\cite{Ziegler-ICM}, and e.g. \cite{Ziegler-semigroup, Sjoberg-Ziegler:semi.alg} for general $d$.
\subsection{Geometric complexity}
Let $\mathcal{F}$ be a family of graded posets of rank $d+1$ with a minimum and a maximum. For instance denote by $\mathcal{F}=\mathcal{P}^d$ (resp. $\mathcal{P}^d_s$) the face lattices of all (resp. simplicial) $d$-polytopes. Let $f(\mathcal{F})$ be the set of $f$-vectors of elements in $\mathcal{F}$, counting the number of elements in each rank $i$, denoted $f_i$, for $1\leq i\leq d$. (Note the shift of index by $1$ with respect to the dimension convention.)

For a subset $T$ of $\mathbb{R}^d$ and $t\in T$ let $\overline{\Conv(T)}$ (resp. $\overline{\Cone_t(T)}$) be the minimal closed convex set (resp. cone with apex $t$) containing $T$. Let $\sigma^d$ denote the $d$-simplex.

The following are geometric consequences of the $g$-theorem.
\begin{theorem}\label{obs:simp.density}
(1) \emph{Convex hull}: $C_d:=\overline{\Conv(f(\mathcal{P}^d_s))}=\overline{\Cone_{f(\sigma^d)}(f(\mathcal{P}^d_s))}$ is a simplicial cone of dimension $\lfloor d/2 \rfloor$.

(2) \emph{Density of rays}: for any $\epsilon>0$ and any $x\in C_d$ there exists a simplicial polytope $P\in \mathcal{P}^d_s$ such that the angle between $x-f(\sigma^d)$ and $f(P)-f(\sigma^d)$ is less than $\epsilon$.

(3) \emph{Density of points}: for any $x\in C_d$ there exists a simplicial polytope $P\in \mathcal{P}^d_s$ such that in the $l_1$-norm $||x-f(P)||_1=O(||x||_1^{1-\frac{1}{\lfloor d/2\rfloor}})=o(||x||_1)$. (The $O(\cdot)$ estimate is tight; $d\ge 2$.)

(4) \emph{Boundary polytopes}: the only polytopes $P\in \mathcal{P}^d_s$ with $f(P)$ on the boundary of $C_d$ are the $k$-stacked polytopes for some $k\le \frac{d}{2} -1$; only the $1$-stacked polytopes have $f(P)$ on an extremal ray, all are on the same ray.
\end{theorem}

When $d\geq 4$, all analogous statements for $\mathcal{P}^d$ seem open.
Explicitly:

\begin{problem}\label{prob:polytope-density}
(1) \emph{Convex hull.} Is $\overline{\Conv(f(\mathcal{P}^d))}=\overline{\Cone_{f(\sigma^d)}(f(\mathcal{P}^d))}$?

(1') \emph{Finite generation.}  Is $\overline{\Cone_{f(\sigma^d)}(f(\mathcal{P}^d))}$ finitely generated?

(2) \emph{Ray density.}
Are the rays from $f(\sigma^d)$ through $f(P)$ for all $P\in \mathcal{P}^d$ dense in $\overline{\Cone_{f(\sigma^d)}(f(\mathcal{P}^d))}$?

(3) \emph{Point density.} Is it true that for any $x\in \overline{\Conv(f(\mathcal{P}^d))}$ there exists $P\in \mathcal{P}^d$ such that $||x-f(P)||_1=o(||x||_1)$?

(4) \emph{Boundary.} For which polytopes $P\in \mathcal{P}^d$ does $f(P)$ lie on the boundary of
$\overline{\Conv(f(\mathcal{P}^d))}$? Of $\overline{\Cone_{f(\sigma^d)}(f(\mathcal{P}^d))}$?
\end{problem}

For $d=4$ Ziegler~\cite{Ziegler-semigroup} showed that
the limits of the rays spanned by $f(\mathcal{P}^d)$ in $\overline{\Cone_{f(\sigma^d)}(f(\mathcal{P}^d))}$ form a convex set; this is open for $d>4$.
Possibly all rays in $\overline{\Cone_{f(\sigma^d)}(f(\mathcal{P}^d))}$ are limit rays, which is equivalent to a YES answer to (1,3); and just to (1) if restricting to the extremal rays.

As for (1') for $d=4$, it is not known if the \emph{fatness} parameter
$\frac{f_1+f_2}{f_0+f_3}$ is bounded above by some constant $C$.
If not, then $\overline{\Cone_{f(\sigma^4)}(f(\mathcal{P}^4))}$ would be determined, with exactly 5 facets~\cite{Bayer-flag_dim_4, Eppstien-Kuperberg-Ziegler:fat}.
Ziegler~\cite{Ziegler-fatness} showed that if $C$ exists then $C\ge 9$.

Similar questions to those in Problem~\ref{prob:polytope-density} can be asked about the set of \emph{flag} $f$-vectors of $d$-polytopes and again are open for $d\ge 4$ (and known for $d\le 3$ by Steinitz~\cite{Steinitz1906}; there the flag $f$-vector is determined by the $f$-vector, see the $cd$-index below).

For $\mathcal{F}$ as above, let $\flag(\mathcal{F})$ be the set of flag $f$-vectors of elements in $\mathcal{F}$, counting the number of chains occupying each subset of ranks $S\subseteq [d]$ (called $S$-chains).
Billera and Ehrenborg~\cite{Billera-Ehrenborg:cdMonotonicity} proved that the simplex $\sigma^d$ minimizes all components of the flag $f$-vector among $d$-polytopes, so we choose it as the apex and consider the cone $\overline{\Cone_{\flag(\sigma^d)}(\flag(\mathcal{P}^d))}$ in the flag analog of Problem~\ref{prob:polytope-density}.

For the larger family of regular CW $(d-1)$-dimensional spheres the situation is better understood.
 Denote by $\mathcal{W}^d$ the family of face posets of regular CW $(d-1)$-spheres, and let $D^d\in \mathcal{W}^d$ be the dihedral $(d-1)$-sphere -- it has exactly two cells in each dimension up to $d-1$; then $D^d$ minimizes the number of $S$-chains for any $S\subseteq [d]$.
Combining a construction of Stanley~\cite{Stanley-cd} with the nonnegativity of the $cd$-index proved by Karu~\cite{Karu-cd}, gives the following known analog of Theorem~\ref{obs:simp.density}(1).

\begin{proposition}\label{obs:CW-density}
 $W_d:=\overline{\Conv(\flag(\mathcal{W}^d))}=\overline{\Cone_{\flag(D^d)}(\flag(\mathcal{W}^d))}$ is a simplicial cone of dimension $c_d-1$, for $c_d$ the $d$th Fibonacci number (e.g. $c_4=5$).
\end{proposition}

The dimension $c_d-1$ was found earlier by Bayer and Billera~\cite{Bayer-Billera:generalizedD-Srelations}, and holds also for the smaller cone $\overline{\Cone_{f(\sigma^d)}(\flag(\mathcal{P}^d))}$; see also~\cite{Kalai:NewBasis}.

The flag analogs of Theorem~\ref{obs:simp.density}(2--4) are open for  $\mathcal{W}^d$, to be
discussed in Sec.~\ref{subsec:cd}.

\subsection{Computational complexity}
Computational complexity gains importance in Enumerative Combinatorics in recent years, see Pak's ICM paper~\cite{Pak-ICM} for a recent survey.
Yet, this perspective is still largely missing in $f$-vector theory.

Fix $d$ and
consider the following decision problems: given a vector $v\in \mathbb{Z}^d_{\ge 0}$ (resp. $v\in \mathbb{Z}^{2^d}_{\ge 0}$), does $v=f(P)$ (resp. $v=\flag(P)$) for some $P\in \mathcal{F}$?

For $\mathcal{F}=\mathcal{P}^d$ this is decidable, by finding all combinatorial types of $d$-polytopes with $n$ vertices -- see Gr\"{u}nbaum's book~\cite[Sec.5.5]{Grunbaum:ConvexPolytopes-03} for a proof using Tarski's elimination of quantifiers theorem.
Using the existential theory of the reals, e.g.~\cite{Canny-exist_reals, Renegar-exist_reals_ICM}, gives an algorithm that runs in time double exponential in size of the encoding of $v$ (in binary, on a deterministic Turing machine).

For $\mathcal{F}=\mathcal{P}_s^d$ this is \emph{effectively} decidable,
namely:
For a vector $v=(v_1,\ldots, v_d)\in \mathbb{Z}^d_{\ge 0}$ denote $N(v):=\sum_{i=1}^{d}\lceil\lg_2(v_i)\rceil$, the number of bits in its encoding in binary. Then,
\begin{theorem}\label{obs:g-effective}
Deciding if $v\in f(\mathcal{P}^d_s)$ can be done in polynomial time in $N(v)$.
\end{theorem}

\begin{problem}\label{prob:decide-f}
Can deciding whether $v\in f(\mathcal{P}^d)$ be done in polynomial time in $N(v)$?
\end{problem}

Recognizing the cone $\overline{\Cone_{f(\sigma^d)}(f(\mathcal{P}^d))}$ may turn out undecidable:
\begin{problem}~\label{prob:cone-comput}
Fix $d\ge 4$. Is the following problem decidable?: given a hyperplane $H$ through $f(\sigma^d)$, does it support the cone $\overline{\Cone_{f(\sigma^d)}(f(\mathcal{P}^d))}$, or contain an interior ray of it?
\end{problem}

As mentioned, for $d=4$, if fatness of $4$-polytopes is unbounded then the decision problem is easy.

The analogs of Problems~\ref{prob:decide-f} and~\ref{prob:cone-comput} for flag-$f$ vectors of $d$-polytopes are open; likewise for Problem~\ref{prob:decide-f} for regular CW $(d-1)$-spheres. Considering the larger family of Gorenstein* posets, their recognition is decidable and we obtain: 
\begin{proposition}\label{obs:decideGorenstein*}
Deciding if $v$ is the flag $f$-vector of a Gorenstein* poset can be done in doubly exponential time in $N(v)$.
\end{proposition}

Is there an effective decision algorithm? In the case $d=4$ the flag $f$-vectors in $\mathcal{W}^4$ are characterized~\cite{Murai-Nevo:rank5cd}; yet it is not clear whether the numerical conditions given can be verified effectively; see Problem~\ref{prob:cd-hardness}.

\section{Preliminaries}
\subsection{$g$-numbers} For $P\in \mathcal{P}^d_s$ with the convention of the introduction, $f_i(P)$ denotes the number of rank $i$ (i.e. $(i-1)$-dimensional) faces of $P$. Define the numbers $h_i(P)$ ($i=0,1\ldots,d$)
by
$$x^d \sum_{i=0}^d h_i(P)(\frac{1}{x})^i =
(x-1)^d \sum_{i=0}^d f_{i}(P)(\frac{1}{x-1})^i.$$
Note that the \emph{$f$-vector} of $P$, $f(P)=(f_0,\ldots,f_{d})$, and its \emph{$h$-vector}
$h(P)=(h_0,h_1,\ldots,h_d)$, are obtained one from the other by applying an invertible linear transformation. Thus, the following theorem indeed characterizes the face numbers of simplicial polytopes.
\begin{theorem}[$g$-theorem~\cite{BilleraLee:SufficiencyMcMullensConditions-1981, Stanley:NumberFacesSimplicialPolytope-80}]\label{thm:g-thm}
An integer vector $h=(h_0,\ldots,h_d)$ is the $h$-vector of
a simplicial $d$-polytope iff the following two conditions hold:

(1) $h_i=h_{d-i}$ for every $0\leq i\leq \lfloor\frac{d}{2}\rfloor$, and

(2) $(h_0=1,h_1-h_0,\ldots,h_{\lfloor\frac{d}{2}\rfloor} - h_{\lfloor\frac{d}{2}\rfloor -1})$ is an M-sequence.
\end{theorem}

\subsection{Gorenstein* posets.}
 A poset $P$ with minimum $\hat{0}$ and maximum $\hat{1}$ is \emph{Gorenstein*} if the reduced order complex $\mathcal{O}(P)$, consisting of all chains in $P\setminus\{\hat{0},\hat{1}\}$, is a Gorenstein* simplicial complex. Namely, for any face $F\in \mathcal{O}(P)$ including the empty one, the link $\lk_{\mathcal{O}(P)}(F)$ has dimension $\dim(\mathcal{O}(P))-|F|$ and is homologous to a rational $(\dim(\mathcal{O}(P))-|F|)$-sphere.
For example, all regular CW spheres are Gorenstein*, thus also all polytope face lattices.

\subsection{$cd$-index.}
For fixed $d$ and $P\in \mathcal{W}^d$, or any Gorenstein* poset of rank $d+1$, we recall its $cd$-index, introduced by Fine.
For a word $w=w_1\cdots w_d$ over alphabet $\{\aaa, \bb\}$, let $S(w):=\{i:\ w_i=b\}$ and for a subset $S\subseteq [d]$ let $w(S)$ be the unique word $w$ over $\{\aaa, \bb\}$ with $d$ letters such that $S(w)=S$.
Define polynomials in non-commuting variables $\Gamma_P(\aaa,\bb):=\sum_{S\subseteq [d]}f_S(P)w(S)$ and $\Psi_P(\aaa,\bb):=\Gamma_P(\aaa - \bb, \bb)$.
It turns out that for $\cc=\aaa+\bb$ of degree $1$ and $\dd=\aaa\bb+\bb\aaa$ of degree $2$, $\Psi_P(\aaa,\bb)=\Phi_P(\cc, \dd)$; this uniquely defined polynomial $\Phi_P$ of homogenous degree $d$ in non-commuting variables $\cc$ and $\dd$ is called the $cd$-index of $P$.
Stanley~\cite{Stanley-cd} proved for $P\in \mathcal{P}^d$, and Karu~\cite{Karu-cd} for any Gorenstein* poset, that:
\begin{theorem}\label{thm:cd>=0}
For any Gorenstein* poset $P$, all coefficients of its $cd$-index $\Phi_P$ are nonnegative.
\end{theorem}
For $B_2$ the boolean lattice on two atoms, $Q_m$ the face poset of the $m$-gon,
and any $cd$-word $w=w_1\cdots w_k$, Stanley~\cite{Stanley-cd} considered the join poset $P_{w,m}=P_1*\ldots* P_k$ where $P_i=B_2$ if $w_i=\cc$ and $P_i=Q_m$ if $w_i=\dd$. It is a regular CW sphere as a join of such. As $\Phi_{P*Q}=\Phi_P\Phi_Q$ holds for any posets $P,Q$ admitting a $cd$-index, Stanley
concluded that when $m$ approaches infinity the coefficient vector of $\Phi_{P_{w,m}}$ approached the ray spanned by the $w$th coordinate.
This explains Proposition~\ref{obs:CW-density}.

\section{Proofs and Discussion}\label{sec:3}
\subsection{Consequences of the $g$-theorem.}
\begin{proof}[Proof of Theorem~\ref{obs:g-effective}]
Recall the $g$-theorem, Theorem~\ref{thm:g-thm}.
Denote $g_i=h_i-h_{i-1}$ for $0<i\le \lfloor d/2\rfloor$.
Checking whether $(1,g_1,\ldots,g_{\lfloor d/2\rfloor})$ is an M-sequence can be done in polynomial time in the size of the encoding of $g:=(g_1,\ldots, g_{\lfloor d/2\rfloor})$ in binary. Indeed, we recall the trivial algorithm one needs to run: (i) for each $i$ produce the $i$th Macaulay representation (see e.g. \cite{Stanley:CombinatoricsCommutativeAlgebra-96} for a definition) of $g_i$ in $\poly(\lg_2 g_i)$-time, (ii) then check if the Macaulay inequalities $g_i^{<i>}\ge g_{i+1}$ hold in $\poly(\max(\lg_2 g_i, \lg_2 g_{i+1})$-time. This verifies Theorem~\ref{obs:g-effective}.
\end{proof}

\begin{proof}[Proof of Theorem~\ref{obs:simp.density}]
The cone $\overline{\Cone_{f(\sigma^d)}(f(\mathcal{P}^d_s))}$ is affinely equivalent to the $g$-cone with apex the origin $\overline{\Cone_0(g(P):\ P\in \mathcal{P}^d_s)}$, which is simply the nonnegative orthant $A_{\lfloor d/2\rfloor}$ in $\mathbb{R}^{\lfloor d/2\rfloor}$. Thus we verify Theorem~\ref{obs:simp.density} by considering the analogous statements for $g$-vectors $g(P)$ and the cone $A_{\lfloor d/2\rfloor}$ rather than $f$-vectors $f(P)$ and the cone $C_d$.

The McMullen-Walkup polytopes~\cite{McMullenWalkup:GLBC-71} approach the extremal rays of $A_{\lfloor d/2\rfloor}$, verifying (1).

For (2), first recall the \emph{connected sum} construction (with respect to a given facet): for two  $d$-polytopes $P_1$ and $P_2$, after applying a projective transformation to one of them, they can be
glued along a common facet (namely, $(d-1)$-face) $\sigma$ to form a new convex $d$-polytope $P=P_1\#_{\sigma}P_2$. Combinatorially, the face lattices are related by
  $\partial P=(\partial P_1\cup_{\sigma}\partial P_2)\setminus \{\sigma\}$.
  Note that on the level of face lattices, the operations $\#_{\sigma}$ are associative and commutative, so we omit the order of summands and of operations from the language.

Now, take connected sum of an appropriate number of copies of appropriate  McMullen-Walkup polytopes  to show that any ray in $A_{\lfloor d/2\rfloor}$ is a limit of a sequence of distinct rays spanned by the $g(P)$ in $A_{\lfloor d/2\rfloor}$.
Indeed, the $g$-vectors sum up under connected sum: $g(P_1\#_{\sigma} P_2)=g(P_1)+ g(P_2)$.

More strongly, for (3) one requires the M-sequence inequalities in the $g$-theorem:
consider the vector $x(a)=(0,0,\ldots,0,a)$ in $A_{\lfloor d/2\rfloor}$ for $a>>1$ and the least $M$-sequence with respect to the reversed lexicographic order such that its $\lfloor d/2 \rfloor$th coordinate equals $a$, denoted
$M(a)$. The Macaulay inequalities show that $||x(a)-M(a)||_1=\Theta(a^{\frac{\lfloor d/2\rfloor -1}{\lfloor d/2\rfloor}})$ (for all $d\ge 2$).
Further, it follows that for any vector $x=(x_1,x_2,\ldots,x_{\lfloor d/2\rfloor})\in A_{\lfloor d/2\rfloor}$ there exists an $M$-sequence $M(x)$ with
$||x-M(x)||_1=O(||x||_1^{\frac{\lfloor d/2\rfloor -1}{\lfloor d/2\rfloor}})$, e.g. by repeating the above argument for the coordinate vectors $x_ie_i$ and summing up.
No better estimate is possible: if $M=(v_1,v_2\ldots,v_{\lfloor d/2\rfloor})\in A_{\lfloor d/2\rfloor}$ is an $M$-sequence with $||x(a)-M||_1=o(a^{1-\frac{1}{\lfloor d/2\rfloor}})$ then $v_{\lfloor d/2\rfloor -1}=o(a^{1-\frac{1}{\lfloor d/2\rfloor}})$ so by the Macaulay inequalities $v_{\lfloor d/2\rfloor}=o(a)$, and we get the contradiction $||M-x(a)||_1\ge |a-o(a)|=\Omega(||x(a)||_1)$.

For (4) consider the Macaulay conditions again. We see that the only $g$-vectors of simplicial $d$-polytopes $g(P)$ on the boundary of $A_{\lfloor d/2\rfloor}$ are those of the form $(a_1,\ldots,a_k,0,\ldots, 0)$ for positive $a_i$s, corresponding exactly to $(k-1)$-stacked polytopes by~\cite{McMullenWalkup:GLBC-71, Murai-Nevo:GLBC}.
The only $g(P)$ on an extremal ray are of the form $(a_1,0,\ldots, 0)$, corresponding to $1$-stacked polytopes by the
Lower Bound Theorem~\cite{Barnette:LBT-73, Kalai:LBT}. This completes the verification of Theorem~\ref{obs:simp.density}.
\end{proof}

\subsection{$cd$-index of regular CW spheres.}\label{subsec:cd}
For a flag analog of Theorem~\ref{obs:simp.density}(2--4) for $W_d$
we first linearly transform to the $cd$-nonnegative orthant $A_{c_d-1}$ in $\mathbb{R}^{c_d-1}$ and consider $cd$-indices rather than flag $f$-vectors.
As for (2),
 we lack the needed connected sum type constructions for complexes in $\mathcal{W}^d$. For example, is it possible to modify a poset $P\in \mathcal{W}^d$ to another poset $P'\in \mathcal{W}^d$ such that (i) $P'$ has a top dimensional cell whose boundary is dihedral (i.e.  isomorphic to $D^{d-1}$), and (ii) the coefficients are close, namely, for any $cd$-word $w$, $|[w]_{\Phi_{P'}}-[w]_{\Phi_P}|=o(\Phi_P(1,1))$?

 If Yes then ray density as asserted in (2) would follow, by taking connected sum over dihedral cells.

As for (4),
the inequality on the $cd$-index by Murai and Yanagawa~\cite[Thm.1.4]{Murai-Yanagawa} shows that the only extremal rays in $W_d$ realized by posets in $\mathcal{W}^d$ are those corresponding to $cd$-words with a single $\dd$.
For $d=4$, thanks to a complete characterization of the possible $cd$-indices~\cite{Murai-Nevo:rank5cd}, (4) is answered. In particular,
on the facet of $A_{c_4-1}$ with coordinate coefficient $[\cc\dd\cc]=0$ the points are sparse. They in fact lie in ``lower dimension", where the $\dd^2$ coordinate is uniquely determined by the coordinates $\cc^2\dd$ and $\dd\cc^2$ via the coefficient equation $[\dd^2]=[\cc^2\dd][\dd\cc^2]$.

Next we consider the  characterization in~\cite{Murai-Nevo:rank5cd} from the computational complexity point of view; the part relevant for a potential computational hardness result is:
\begin{theorem}[Murai-N.~\cite{Murai-Nevo:rank5cd}]
Let $\Phi$ be a $cd$-polynomial of homogenous degree $4$ with nonnegative integer coefficients satisfying $[\cc^4]=1$ and $[\cc\dd\cc]=1$. Then $\Phi=\Phi(P)$ for some Gorenstein* poset of rank $5$ (or even regular CW $3$-sphere) iff there exist nonnegative integers $x_1,x_2,x_3$ and $y_1,y_2,y_3$ such that
\begin{equation}\label{eq:cdc=1}
x_1+x_2+x_3=[\cc^2\dd],\  y_1+y_2+y_3=[\dd \cc^2],\  x_1y_1+x_2y_2+x_3y_3=[\cc^2\dd][\dd \cc^2]-[\dd^2].                                                                                                                                                   \end{equation}
\end{theorem}
\begin{problem}\label{prob:cd-hardness}
Let $N=\lceil\lg_2 [\cc^2\dd]\rceil +\lceil\lg_2 [\dd\cc^2]\rceil +\lceil\lg_2 [\dd^2]\rceil$.
Can it be decided in $\poly(N)$-time whether the diophantine system Eq.~(\ref{eq:cdc=1}) has a solution?
\end{problem}
Deciding in $\exp(N)$-time is trivial. Recall that some binary diophantine quadratics are known to be NP-complete~\cite{Manders-Adleman:NPcompleteQuadrics}.
Next we consider arbitrary Gorenstein* posets.
\begin{proof}[Proof of Proposition~\ref{obs:decideGorenstein*}]
All Gorenstein* posets of fixed rank with given flag $f$-vector of binary bit complexity $N$ can be obtained in $\exp(\exp(N))$-time. Indeed, the total number of possible chains of faces is $\Pi_i f_{\{i\}}=\exp(O(N))$, and each potential poset
corresponds to a subset, so all together we have to consider $\exp(\exp(O(N)))$ number of posets $P$. For each
 $P$ we compute the cellular homology groups over say the field of rationals, for all intervals in $P$; they are $\exp(O(N))$ many. For each interval the computation is polynomial in the size of the interval, so takes $\poly(\exp(O(N)))$ time.
 Proposition~\ref{obs:decideGorenstein*} follows.
 \end{proof}
Possibly a decision can be made in $\poly(N)$-time.

\textbf{Acknowledgements.}
I thank the anonymous referees for helpful suggestions on the presentation.
\bibliographystyle{plain}
\bibliography{biblioB}

\end{document}